\documentclass[11pt]{article}
\usepackage{amsmath}
\usepackage{amssymb}
\usepackage{amsthm}
\usepackage{xcolor}
\usepackage[mathscr]{eucal}
\input{epsf}
\theoremstyle{definition}
\newtheorem{definition}{Definition}
\newtheorem{theorem}[definition]{Theorem}

\newtheorem{lemma}[definition]{Lemma}
\newtheorem{corollary}[definition]{Corollary}
\theoremstyle{remark}
\newtheorem{remark}[definition]{Remark}

\newcounter{enumctr}

\newcommand{\R}{\mathbb{R}}

\newcommand{\eps}{\varepsilon}
\renewcommand{\phi}{\varphi}
\newcommand{\E}{\mathbb{E}}
\newcommand{\mP}{\mathbb{P}}

\newcommand{\cF}{\mathcal{F}}

\renewcommand{\phi}{\varphi}

\newcommand{\rT}{\mathrm {T}}
\begin{document}


\title{Asymptotic separation between solutions of Caputo fractional stochastic differential equations}
\author{T.S. Doan
\footnote{Email: dtson@math.ac.vn, Institute of Mathematics, Vietnam Academy of Science and Technology, 18 Hoang Quoc Viet, Ha Noi, Viet Nam.}
, P. T. Huong\footnote{Email: pthuong175@gmail.com,  Le Quy Don Technical University, 236 Hoang Quoc Viet, Ha Noi, Viet Nam.}
, Peter. E. Kloeden 
\footnote{Email: kloeden@math.uni-frankfurt.de,  School of Mathematics and Statistics, Huazhong University of Science \& Technology,
Wuhan 430074, China.}
 and
 H. T. Tuan\footnote{Email: httuan@math.ac.vn, Institute of Mathematics, Vietnam Academy of Science and Technology, 18 Hoang Quoc Viet, Ha Noi, Viet Nam.}
}
\maketitle
\begin{abstract}
Using a temporally weighted norm we first establish a result on the global existence and uniqueness of  solutions for  Caputo fractional stochastic differential equations of order $\alpha\in(\frac{1}{2},1)$ whose coefficients satisfy a standard Lipschitz condition. For this class of systems  we then show that the asymptotic distance between two distinct solutions is greater than $t^{-\frac{1-\alpha}{2\alpha}-\eps}$ as $t \to \infty$ for any $\eps>0$. As a consequence,  the mean square Lyapunov exponent of an arbitrary non-trivial solution of a bounded linear Caputo fractional stochastic differential equation is always non-negative.
\end{abstract}

Fractional stochastic differential equations, Existence and uniqueness solutions, Temporally weighted norm,  Continuous dependence on the initial condition, Asymptotic behavior, Lyapunov exponents



\section{Introduction}
Fractional differential equations is now receiving an increasing attention due to their applications in a variety of disciplines such as mechanics, physics, chemistry, biology, electrical engineering, control theory, viscoelasticity, heat conduction in materials with memory, amongst others. 
For more details, we refer the interested reader to the monographs  \cite{Kamal},    \cite{Podlubny},  \cite{Oldham},   \cite{Diethelm},   \cite{Miller},    \cite{Samko}      and the references therein.

In contrast to the huge  number of publications in deterministic fractional differential equations,  there  have been only a few papers dealing with stochastic differential equations involving with a Caputo fractional time derivative and most of these articles have  attempted  to establish a result  on the existence and uniqueness of solutions. Here, we distinguish two type of solutions:

 The first one is mild solutions and we refer the reader to  \cite{Benchaabane1_17,Sakthivel_13}  for a result on the existence and uniqueness of this type of solutions.

 The other type of solution is defined as a solution of an associated stochastic integral equations  and as far as we are aware the only reference for the question of existence and uniqueness of this type of solutions is \cite{Kloeden_16}. However, in this paper there is a gap in the argument about the successive approximation method for extending the result about the local existence and uniqueness of solutions to the global existence and uniqueness of solutions. This gap comes from the fact that the kernel in the stochastic integral equations depends on time and this fact is a characteristic property of Caputo fractional systems, see Remark \ref{Remark1} for a more detail.

To fill in the gap mentioned above in \cite{Kloeden_16}, a special weighted norm called Bielecki type norm is introduced. With respect to this norm, we are able to prove that the operator associated with the stochastic integral equation is globally contractive and its fixed point gives rise to the appropriate  global solution of the system. Furthermore, we also show that the solutions depend continuously on the initial values.

After showing the existence of global solutions, our interest is to investigate the asymptotic behavior of solutions. Our contribution in this direction is to establish a lower bound on the asymptotic distance between two distinct solutions of a fractional stochastic differential equation. As a consequence, we show that the mean square Lyapunov exponent of an arbitrary non-trivial solution of a bounded linear Caputo fractional stochastic differential equation is always non-negative. This surprising fact was  obtained   in the deterministic fractional differential equations  \cite{Cong_14}.

The paper is structured as follows: In Section \ref{Section2}, we introduce briefly about Caputo fractional stochastic differential equations and state the main results of the paper. The first part of Section \ref{Section3} is devoted to show the result on the global existence and uniqueness of solutions (Theorem \ref{main_result_1}). The second main result (Theorem \ref{main_result2}) concerning a lower bound on the asymptotic separation of solutions is proved in the second part of Section \ref{Section3}.
\section{Preliminaries and the statement of the main results}\label{Section2}
Consider a Caputo fractional stochastic differential equation (for short Caputo FSDE) of order $\alpha\in(\frac{1}{2},1)$ of the following form
\begin{equation}\label{MainEq_1}
^{\!C}D^{\alpha}_{0+} X(t)=b(t,X(t)) + \sigma(t, X(t))\;\frac{dW_t}{dt},
\end{equation}
where $b,\sigma:[0,\infty)\times \R^d\rightarrow \R^{d}$, $\sigma:[0,\infty)\times \R^d\rightarrow \R^{d}$ are measurable and $(W_t)_{t\in[0,\infty)}$ is a standard scalar Brownian motion on an underlying complete filtered probability space $(\Omega,\cF,\mathbb{F}:=\{\cF_t\}_{t\in [0,\infty)},\mP)$. For each $t\in[0,\infty)$, let $\frak X_t:=\mathbb{L}^2(\Omega,\cF_t,\mP)$ denote the space of all  $\cF_t$-measurable, mean square integrable functions $f=(f_1,,\dots,f_d)^{\rT}:\Omega\rightarrow \R^d$ with
\[
\|f\|_{\rm ms}:=\sqrt{\sum_{i=1}^d \mathbb{E}(|f_i|^2)}=\sqrt{\E \|f\|^2},
\]
where $\R^d$ is endowed with the standard Euclidean norm. A process $X:[0,\infty)\rightarrow \mathbb{L}(\Omega,\cF,\mP)$ is said to be $\mathbb{F}$-adapted if $X(t)\in \frak X_t$ for all $t\geq 0$.
For each $\eta\in \frak X_0$, a $\mathbb{F}$-adapted process $X$ is called a solution of \eqref{MainEq_1} with the initial condition $X(0)=\eta$ if the following equality holds for $t\in [0,\infty)$
\begin{equation}\label{IntegrableForm}
X(t)=\eta+\frac{1}{\Gamma(\alpha)}
\left(
\int_0^t (t-\tau)^{\alpha -1} b(\tau,X(\tau))\;d\tau+
\int_0^t (t-\tau)^{\alpha -1} \sigma(\tau,X(\tau))\;dW_\tau\right),
\end{equation}
where $\Gamma(\alpha):=\int_0^\infty \tau^{\alpha-1}\exp{(-\tau)}\;d\tau$ is the Gamma function, see \cite[p. 209]{Kloeden_16}. In the remaining of the paper, we assume that the coefficients $b$ and $\sigma$ satisfy the following  standard conditions:
\begin{itemize}
\item [(H1)] There exists $L>0$ such that for all $x,y\in\R^d$, $t\in [0,\infty)$
\begin{equation*}\label{L_cond}
\|b(t,x)-b(t,y)\|+ \|\sigma(t,x)-\sigma(t,y)\|\leq L\|x-y\|.
\end{equation*}
\item [(H2)]  $\sigma(\cdot,0)$  is essentially bounded, i.e.
\[
\|\sigma(\cdot,0)\|_{\infty}:=\mbox{ess}\hspace{-3mm}\sup_{\tau\in [0,\infty)}\|\sigma(\tau,0)\|<\infty
\]
 and $b(\cdot,0)$ is $\mathbb{L}^2$ integrable, i.e.
\[
\int_0^\infty
\|b(\tau,0)\|^2\;d\tau<\infty.
\]
\end{itemize}
Our first result in this paper is to show the global existence and uniqueness solutions of \eqref{MainEq_1} when (H1) and (H2) hold. Furthermore, we also show the continuity dependence of solutions on the initial values.
\begin{theorem}[Global existence and uniqueness \& Continuity dependence on the initial values of  solutions of Caputo FSDE]\label{main_result_1}
Suppose that {\textup{(H1)}} and {\textup{(H2)}} hold.  Then  
\begin{itemize}
\item [(i)]  for any $\eta\in \frak X_0$, the initial value problem \eqref{MainEq_1} with the initial condition $X(0)=\eta$ has a unique global solution on the whole interval $[0,\infty)$ denoted by $\varphi(\cdot,\eta)$;
\item [(ii)]  on  any bounded time interval $[0,T]$, where $T>0$, the solution $\phi(\cdot,\eta)$ depends continuously on $\eta$, i.e.
\[
\lim_{\zeta\to \eta} \sup_{t\in [0,T]}\|\phi(t,\zeta)-\phi(t,\eta)\|_{\rm ms}=0.
\]
\end{itemize}
\end{theorem}
Our next result is to establish a lower bound on the asymptotic separation between two distinct solutions of \eqref{MainEq_1}.
\begin{theorem}\label{main_result2}
Let $\eta,\zeta\in  \frak X_0$ such that $\eta\not=\zeta$. Then, for any $\eps>0$, 
\begin{equation*}\label{convergence rate}
\limsup_{t\to\infty}\;t^{\frac{1-\alpha}{2\alpha}+\eps}\|\varphi(t,\eta)-\varphi(t,\zeta)\|_{\rm{ms}}=\infty.
\end{equation*}
\end{theorem}
Finally, we give an application of the main results concerning the mean square Lyapunov exponent of non-trivial solutions to a bounded bilinear Caputo  FSDE. To formulate this result, we consider the following equation
\begin{equation}\label{linearEq}
^{C}D^\alpha_{0+}x(t)=A(t)x(t)+B(t)x(t)\frac{dW_t}{dt},
\end{equation}
where $A,B:[0,\infty)\rightarrow \R^{d\times d}$ are measurable and essentially bounded, i.e.
$$
\mbox{ess}\hspace{-2mm}\sup_{t\in [0,\infty)}\|A(t)\|, \,\,  \mbox{ess}\hspace{-2mm}\sup_{t\in [0,\infty)}\|B(t)\|<\infty.
$$
 By virtue of Theorem \ref{main_result_1}, for each  $\eta\in \frak X_0\setminus\{0\}$, there exists a unique the solution of \eqref{linearEq}, denoted by $\Phi(\cdot,\eta)$, satisfying the initial condition $X(0)=\eta$. The \emph{mean square Lyapunov exponent} of $\Phi(\cdot,\eta)$ is defined by
\begin{equation}\label{Mean-SquareSpectrum}
\lambda_{\rm{ms}}(\Phi(\cdot,\eta))
:=
\limsup_{t\to\infty}\frac{1}{t}\log\|\Phi(t,\eta)\|_{\rm ms},
\end{equation}
see e.g. \cite{Son_15}. In the following corollary, we show the non-negativity of the mean square Lyapunov exponent of an arbitrary non-trivial solution.
\begin{corollary}[Non-negativity of mean square Lyapunov exponent for solutions of linear Caputo fsde]\label{Cor1}
The mean square Lyapunov exponent of a nontrivial solution of \eqref{linearEq} is always non-negative, i.e.
\[
\lambda_{\rm{ms}}(\Phi(\cdot,\eta)) \ge 0\quad \hbox{for all}\;\; \eta\in \frak X_0\setminus\{0\}.
\]
\end{corollary}
\begin{proof}
Let $\eta\in \frak X_0\setminus\{0\}$ be arbitrary. Using Theorem \ref{main_result2},  we obtain 
\[
\limsup_{t\to\infty}t^{\frac{1-\alpha}{2\alpha}+\varepsilon}\|\Phi(t,\eta)\|_{\rm{ms}}=\infty,
\]
where $\eps>0$ is arbitrary. Hence, there exists $T>0$ such that
\[
\|\Phi(t,\eta)\|_{\rm{ms}}\geq t^{-\left(\frac{1-\alpha}{2\alpha}+\varepsilon\right)}\qquad \hbox{for all}\; t\geq T,
\]
which together with \eqref{Mean-SquareSpectrum} implies that
\[
\lambda_{\rm{ms}}(\Phi(\cdot,\eta))
\geq
\limsup_{t\to\infty}\frac{1}{t}\log \left( t^{-\left(\frac{1-\alpha}{2\alpha}+\varepsilon\right)}\right)
= 0.
\]
\end{proof}

\section{Proof of the main results}\label{Section3}
\subsection{Existence, uniqueness and continuity dependence on the initial values of solutions}
Our aim in this subsection is to prove the result on global existence, uniqueness and continuity dependence on the initial values of solutions to the equation \eqref{MainEq_1}. In fact, in order to prove Theorem \ref{main_result_1}(i) it is equivalent to show the existence and uniqueness solutions on an arbitrary interval $[0,T]$, where $T>0$ is arbitrary. In what follows we choose and fix a $T>0$ arbitrarily. 

 Let $\mathbb{H}^2([0,T])$ be the space of all the processes $X$ which are measurable, $\mathbb{F}_T$-adapted, where $\mathbb{F}_T:=\{\cF_t\}_{t\in [0,T]}$, and satisfies that
$$
\|X\|_{\mathbb{H}^2}:=\sup_{0\le t\le T}\|X(t)\|_{\rm ms}<\infty.
$$
Obviously, $(\mathbb{H}^2([0,T]),\|\cdot\|_{\mathbb{H}^2})$ is a Banach space. For any $\eta\in \mathfrak{X}_0$, we define an operator $\mathcal{T}_\eta: \mathbb{H}^2([0,T])\rightarrow \mathbb{H}^2([0,T])$ by
\begin{equation}\label{Operator}
T_\eta\xi(t)
:=
\eta+\frac{1}{\Gamma(\alpha)}
\left(\int_0^t (t-\tau)^{\alpha -1} b(\tau,\xi(\tau))\;d\tau+
\int_0^t (t-\tau)^{\alpha -1} \sigma(\tau,\xi(\tau))\;dW_\tau\right).
\end{equation}
The following lemma is devoted to showing that   this operator is well-defined.
%
%
%
%
\begin{lemma}\label{Welldefine}
For any $\eta\in \frak{X}_0$, the operator $\mathcal{T}_\eta$ is well-defined.
\end{lemma}
\begin{proof}
Let $\xi\in \mathbb{H}^2([0,T])$ be arbitrary. From the definition of $\mathcal{T}_\eta\xi$ as in \eqref{Operator} and the inequality $\|x+y+z\|^2\leq 3 (\|x\|^2+ \|y\|^2+ \|z\|^2)$ for all $x,y,z\in \R^d$, we have  for all $t\in [0,T]$
\begin{equation}\label{FirstEstimate}
\begin{array}{ll}
\left\|\mathcal{ T}_\eta \xi(t)\right\|_{\rm ms}^2
\le & 3 \|\eta\|_{\rm ms}^2 + \frac{3}{{\Gamma {{\left( \alpha  \right)}^2}}} \mathbb E{\left( \left\|{\int_0^t {{{\left( {t - \tau } \right)}^{\alpha  - 1}} {b\left( {\tau ,\xi \left( \tau  \right)} \right)}d\tau } }\right\|^2 \right)}\\[1ex]
&  + \frac{3}{{\Gamma {{\left( \alpha  \right)}^2}}} \mathbb E{\left(\left\| {\int\limits_0^t {{{\left( {t - \tau } \right)}^{\alpha  - 1}}{\sigma \left( {\tau ,\xi\left( \tau  \right)} \right)}d{W_\tau }} }\right\|^2 \right)}.
\end{array}
\end{equation}
By the H\"{o}lder inequality, we obtain
\begin{eqnarray}
\mathbb E{\left(\left\| {\int\limits_0^t {{{\left( {t - \tau } \right)}^{\alpha  - 1}}b\left( {\tau ,\xi\left( \tau  \right)} \right)d\tau } }\right\|^2 \right)}
&\leq &
 \int_0^t {{{\left( {t - \tau } \right)}^{2\alpha  - 2}}d\tau }\; \mathbb E\left( {\int_0^t {{{\left\| {b\left( {\tau ,\xi\left( \tau  \right)} \right)} \right\|}^2}d\tau } } \right)\notag\\
&=&
\frac{t^{2\alpha  - 1}}{2\alpha  - 1} \mathbb E\left( \int_0^t {{{\left\| {b\left( {\tau ,\xi\left( \tau  \right)} \right)} \right\|}^2}d\tau }  \right).\label{tam_1}
 \end{eqnarray}
From $\textup{(H1)}$, we derive 
\begin{eqnarray*}
\|b(\tau,\xi(\tau))\|^2
&\leq&
 2\left( \|b(\tau,\xi(\tau))-b(\tau,0)\|^2+ \|b(\tau,0)\|^2\right)\\[1.5ex]
&\leq &
2 L^2 \|\xi(\tau)\|^2+ 2 \|b(\tau,0)\|^2.
\end{eqnarray*}
Therefore,
\begin{eqnarray*}
\mathbb E\left( \int_0^t \left\| {b\left( {\tau ,\xi\left( \tau  \right)} \right)} \right\|^2d\tau   \right)
& \leq &
2L^2 \mathbb{E}\left(\int_0^t {{{\left\| {\xi\left( \tau  \right)} \right\|}^2}}\,d\tau \right)
+ 2 \int_0^t \|b(\tau,0)\|^2\;d\tau\\
&\leq &
2L^2 T \sup _{t \in \left[ {0,T} \right]} \mathbb E({\left\| {\xi\left( t  \right)} \right\|^2})
+
2 \int_0^T \|b(\tau,0)\|^2\;d\tau
\end{eqnarray*}
which together with \eqref{tam_1} implies that
\begin{equation}\label{oper_1}
\mathbb E{\left(\left\| {\int_0^t {{{\left( {t - \tau } \right)}^{\alpha  - 1}}b\left( {\tau ,\xi\left( \tau  \right)} \right)d\tau } }\right\|^2 \right)}
\leq
 \frac{{2{L^2}{T^{2\alpha }}}}{{2\alpha  - 1}} \|\xi\|_{\mathbb H^2}^2 + \frac{2{T^{2\alpha-1 }}}{2\alpha-1} \int_0^T \|b(\tau,0)\|^2\;d\tau.
\end{equation}
Now, using Ito's isometry (see e.g. \cite[p. 87]{Kloeden_92} ), we obtain 
\begin{align*}
\notag\mathbb E{\left(\left\| {\int\limits_0^t {{{\left( {t - \tau } \right)}^{\alpha  - 1}}{\sigma \left( {\tau ,\xi\left( \tau  \right)} \right)} d{W_\tau }} }\right\|^2 \right)}&=\sum_{1\leq i\leq d}\mathbb{E}\left(\int_0^t (t-\tau)^{\alpha-1}\sigma_i(\tau,\xi(\tau))dW_\tau\right)^2\\
&=\sum_{1\leq i\leq d}\mathbb{E}\left(\int_0^t (t-\tau)^{2\alpha-2}|\sigma_i(\tau,\xi(\tau))|^2 d\tau\right)\\
&= \mathbb E\int\limits_0^t {{{\left( {t - \tau } \right)}^{2\alpha  - 2}}{{\left\| {\sigma \left( {\tau ,\xi\left( \tau  \right)} \right)} \right\|}^2}d\tau }.
\end{align*}
From (H1), we also have
\begin{eqnarray*}
\|\sigma(\tau,\xi(\tau))\|^2
\leq
2 L^2 \|\xi(\tau)\|^2+ 2 \|\sigma(\tau,0)\|^2 \leq 2 L^2 \|\xi(\tau)\|^2+ 2 \|\sigma(\cdot,0)\|^2_{\infty}.
\end{eqnarray*}
Therefore, for all $t\in [0,T]$ we have
\begin{eqnarray*}
&& \mathbb E\left(\left\|{\int_0^t {{{\left( {t - \tau } \right)}^{\alpha  - 1}}{\sigma \left( {\tau ,\xi\left( \tau  \right)} \right)} d{W_\tau }} }\right\|^2 \right)\\
&\leq &
2{L^2} \mathbb E\int\limits_0^t {{\left( {t - \tau } \right)}^{2\alpha  - 2}}{{\left\| {\xi\left( \tau  \right)} \right\|}^2}\;d\tau
+ \;2 \|\sigma(\cdot,0)\|^2_{\infty} \int_0^t {{\left( {t - \tau } \right)}^{2\alpha  - 2}}\; d\tau \\
&\leq&
2{L^2}\frac{{{T^{2\alpha  - 1}}}}{{2\alpha  - 1}} \|\xi\|_{\mathbb H_2}^2 +
\frac{2 T^{2\alpha  - 1}}{2\alpha  - 1} \left\| {\sigma \left( {\cdot ,0} \right)} \right\|^2_{\infty}.
\end{eqnarray*}
This together with \eqref{FirstEstimate} and \eqref{oper_1} implies that $\left\| \mathcal{T}_\eta \xi\right\|_{\mathbb H^2}<\infty$. Hence, the map $\mathcal{T}_{\eta}$ is well-defined.
\end{proof}
To prove existence and uniqueness of solutions,we will show that the operator $\mathcal{T}_\eta$ defined as above is contractive under a suitable 
temporally weighted norm (cf. \cite[Remark 2.1]{Han_Kloeden} for the same method to prove the existence and uniqueness of solutions of stochastic differential equations). Here, the weight function is the Mittag-Leffler function $E_{2\alpha-1}(\cdot)$ defined as:
\[
E_{2\alpha-1}(t):=\sum_{k=0}^{\infty}\frac{t^k}{\Gamma((2\alpha-1)k+1)}\quad \hbox{for all}\; t\in\R.
\]
For more details on the Mittag-Leffler functions we  refer the reader to the book \cite[p. 16]{Podlubny}. The following result is a technical lemma which is used later to estimate the operator $\mathcal{T}_\eta$.
\begin{lemma}\label{ML_estimat}
For any $\alpha>\frac{1}{2}$ and $\gamma>0$, the following inequality holds:
\[
\frac{\gamma}{{\Gamma \left( {2\alpha  - 1} \right)}}\int_0^t {{{\left( {t - \tau } \right)}^{2\alpha  - 2}}{E_{2\alpha  - 1}}\left( {\gamma {\tau ^{2\alpha  - 1}}} \right)d\tau }  \le {{E_{2\alpha  - 1}}\left( {\gamma {t ^{2\alpha  - 1}}} \right)}.
\]
\end{lemma}
\begin{proof}
	Let $\gamma>0$ be arbitrary. Consider the corresponding linear Caputo fractional differential equation of the following form
\begin{equation}\label{ML}
^{\!C}D^{2\alpha-1}_{0+} x(t)=\gamma x(t).
\end{equation}
The Mittag-Leffler function $E_{2\alpha-1}(\gamma t^{2\alpha-1})$ is a solution of \eqref{ML}, see e.g. \cite[p. 135]{Diethelm}. Hence, the following equality holds:
\[
E_{2\alpha-1}(\gamma t^{2\alpha-1})
=
1
+
\frac{\gamma}{\Gamma(2\alpha-1)}
\int_0^t (t-\tau)^{2\alpha-2} E_{2\alpha-1}(\gamma \tau^{2\alpha-1})\;d\tau,
\]
which completes the proof.
\end{proof}
We are now in a position to prove Theorem \ref{main_result_1}.
\begin{proof}[Proof of Theorem \ref{main_result_1}]
Let $T>0$ be arbitrary. Choose and fix a positive constant $\gamma$ such that
\begin{equation}\label{coefficient_weight}
\gamma>\frac{3L^2(T+1)\Gamma(2\alpha-1)}{\Gamma(\alpha)^2}.
\end{equation}
On the space $\mathbb{H}^2([0,T])$, we define a weighted norm $\|\cdot\|_{\gamma}$ as below
\begin{equation}\label{Weightnorm}
\|X\|_{\gamma}:=\sup_{t\in [0,T]}\sqrt{
\frac{\mathbb{E}(\|X(t)\|^2)}{E_{2\alpha -1}(\gamma t^{2\alpha-1})}}\qquad \hbox{for all }  X\in \mathbb{H}^2([0,T]).
\end{equation}
Obviously, two norms $\|\cdot\|_{\mathbb{H}^2}$ and $\|\cdot\|_{\gamma}$ are equivalent. Thus, $(\mathbb{H}^2([0,T]),\|\cdot\|_\gamma)$ is also a Banach space.\\

\noindent (i) Choose and fix $\eta\in \mathfrak X_0$. By virtue of Lemma \ref{Welldefine}, the operator $\mathcal{T}_\eta$ is well-defined. We will prove that the map $\mathcal{T}_{\eta}$ is contractive with respect to the norm $\|\cdot\|_{\gamma}$.

For this purpose, let $\xi,\widehat{\xi}\in \mathbb{H}^2([0,T])$ be arbitrary. From \eqref{Operator} and the inequality $\|x+y\|^2\leq 2 (\|x\|^2+ \|y\|^2)$ for all $x,y\in \R^d$, we derive the following inequalities for all $t\in [0,T]$:
\begin{eqnarray*}
&& \mathbb E\left( {{{\left\| {\mathcal{T}_{{\eta}}{\xi\left( t \right)}  - {\mathcal{T}_{{\eta}}} {\widehat \xi\left( t \right)}} \right\|}^2}} \right)\\
&\leq &	\frac{2}{{\Gamma {{\left( \alpha  \right)}^2}}}\mathbb E{\left(\left\| {\int\limits_0^t {{{\left( {t - \tau } \right)}^{\alpha  - 1}}\big({b\left( {\tau ,\xi\left( t \right)} \right) - b( {\tau ,\widehat \xi\left( t \right)})} \big)d\tau } }\right\|^2 \right)}\\
&& + \frac{2}{{\Gamma {{\left( \alpha  \right)}^2}}}\mathbb E{\left(\left\| {\int\limits_0^t {{{\left( {t - \tau } \right)}^{\alpha  - 1}}\big( {\sigma \left( {\tau ,\xi\left( t \right)} \right) - \sigma ( {\tau ,\widehat \xi\left( t \right)})} \big)d{W_\tau }} }\right\|^2 \right)}.
\end{eqnarray*}
Using the H\"{o}lder inequality and $(\textup{H1})$, we obtain 
\begin{eqnarray*}
&&\mathbb E{\left(\left\| {\int\limits_0^t {{{\left( {t - \tau } \right)}^{\alpha  - 1}}\big( {b( {\tau ,\xi\left( \tau \right)}) - b( {\tau ,\widehat \xi\left( \tau \right)} )} \big)d\tau } }\right\|^2 \right)}\\
&\le & {L^2}t\; \int_0^t {\left( {t - \tau } \right)}^{2\alpha  - 2}{{\mathbb E(\| {\xi\left( \tau \right) - \widehat \xi\left( \tau \right)} \|}^2)}\;d\tau .
\end{eqnarray*}
On the other hand, by Ito's isometry and $\textup{(H1)}$, we have
\begin{eqnarray*}
&& \mathbb E{\left(\left\| {\int\limits_0^t {{{\left( {t - \tau } \right)}^{\alpha  - 1}}\big( {\sigma \left( {\tau ,\xi\left( \tau \right)} \right) - \sigma ( {\tau ,\widehat \xi\left( \tau \right)} )} \big)d{W_\tau }} }\right\|^2 \right)}\\
& =&  \mathbb E\int_0^t {{\left( {t - \tau } \right)}^{2\alpha  - 2}}{{\| {\sigma \left( {\tau ,\xi\left( \tau \right)} \right) - \sigma ( {\tau ,\widehat \xi\left( \tau \right)})} \|}^2}\;d\tau \\
& \le &
{L^2} \int_0^t {{\left( {t - \tau } \right)}^{2\alpha  - 2}}{{\mathbb E(\| {\xi\left( \tau \right) - \widehat \xi\left( \tau \right)} \|}^2)}\;d\tau .
\end{eqnarray*}
Thus, for all $t\in [0,T]$ we have
\begin{eqnarray*}
	\E\left( {\left\| T_{\eta} {\xi\left( t \right)} - T_{\eta} {\widehat \xi( t )} \right\|^2} \right)
	&\le &
\frac{2L^2 (t+1)}{\Gamma(\alpha)^2}
	\int_0^t {{\left( {t - \tau } \right)}^{2\alpha  - 2}}{{\mathbb E(\| {\xi\left( \tau \right) - \widehat \xi\left( \tau \right)} \|}^2)}\;d\tau,
\end{eqnarray*}
which together with the definition of $\|\cdot\|_{\gamma}$ as in \eqref{Weightnorm} implies that
\[
	\frac{\E\left( {\left\| T_{\eta}{\xi\left( t \right)} - T_{\eta} {\widehat \xi( t )} \right\|^2} \right)}{E_{2\alpha-1}(\gamma t^{2\alpha-1})}
\leq
\frac{2L^2 (t+1)}{\Gamma(\alpha)^2}
\frac{\int_0^t (t-\tau)^{2\alpha-2} E_{2\alpha-1}(\gamma \tau^{2\alpha-1})\;d\tau}{ E_{2\alpha-1}(\gamma t^{2\alpha-1})} \|\xi-\widehat \xi\|_{\gamma}^2.
\]
In light of Lemma \ref{ML_estimat}, we have for all $t\in[0,T]$
\[
\frac{\E\left( {\left\| T_{\eta}{\xi\left( t \right)} - T_{\eta} {\widehat \xi( t )} \right\|^2} \right)}{E_{2\alpha-1}(\gamma t^{2\alpha-1})}
\leq
\frac{2 \Gamma(2\alpha-1) L^2 (T+1)}{ \Gamma(\alpha)^2 \gamma} \|\xi-\widehat \xi\|_{\gamma}^2.
\]
Consequently,
\[
\|T_{\eta}\xi- T_{\eta}\widehat\xi\|_{\gamma}
\leq
\kappa \|\xi-\widehat \xi\|_{\gamma},
\qquad\hbox{where } \kappa:=\sqrt{\frac{2 \Gamma(2\alpha-1) L^2 (T+1)}{ \Gamma(\alpha)^2 \gamma} }.
\]
By \eqref{coefficient_weight}, we have $\kappa<1$ and therefore the operator $\mathcal{T}_{\eta}$ is a contractive map on $(\mathbb{H}^2([0,T]),\|\cdot\|_{\gamma})$. Using the Banach fixed point theorem, there exists a unique fixed point of this map in $\mathbb{H}^2([0,T])$. This fixed point is also the unique solution of \eqref{MainEq_1} with the initial condition $X(0)=\eta$. The proof of this part is complete.\\

\noindent (ii) Choose and fix $T>0$ and $\eta,\zeta\in\mathfrak{X}_0$. Since $\phi(\cdot,\eta)$ and $\phi(\cdot,\zeta)$ are solutions of \eqref{MainEq_1} it follows that
\begin{eqnarray*}
\phi(t,\eta)-\phi(t,\zeta)
=\eta-\zeta &+& \frac{1}{\Gamma(\alpha)}\int_0^t (t-\tau)^{\alpha-1} (b(\tau,\phi(\tau,\eta))-b(\tau,\phi(\tau,\zeta)))\;d\tau\\
&+&
\frac{1}{\Gamma(\alpha)}\int_0^t (t-\tau)^{\alpha-1} (\sigma(\tau,\phi(\tau,\eta))-\sigma(\tau,\phi(\tau,\zeta)))\;dW_\tau
\end{eqnarray*}
Hence, using  the inequality $\|x+y+z\|^2\leq 3 (\|x\|^2+ \|y\|^2+ \|z\|^2)$ for all $x,y,z\in \R^d$, (H1), the H\"{o}lder inequality and Ito's isometry (see Part (i)), we obtain  
\begin{eqnarray*}
\mathbb E\left( \| \varphi(t,\eta)-\varphi(t,\zeta)\|^2 \right)
&\leq& \frac{3L^2(t + 1)}{\Gamma( \alpha)^2}\int_0^t ( t - \tau)^{2\alpha  - 2}\mathbb E(\| \varphi( \tau,\eta) - \varphi(\tau,\zeta) \|^2)\;d\tau\\
&& +3\mathbb{E}(\|\eta-\zeta\|^2).
\end{eqnarray*}
By definition of $\|\cdot\|_{\gamma}$, we have
\begin{eqnarray*}
\frac{\E\left( \| \varphi(t,\eta)-\varphi(t,\zeta)\|^2 \right)}{E_{2\alpha-1}(\gamma t^{2\alpha-1})}
&\leq &
 \frac{3L^2(t + 1)}{\Gamma( \alpha)^2}\frac{\int_0^t ( t - \tau)^{2\alpha  - 2}E_{2\alpha-1}(\gamma \tau^{2\alpha-1})\;d\tau}{E_{2\alpha-1}(\gamma t^{2\alpha-1})}\times\\
 &&\|\phi(\cdot,\eta)-\phi(\cdot,\zeta)\|_{\gamma}^2 + 3\mathbb{E}(\|\eta-\zeta\|^2).
\end{eqnarray*}
By virtue of Lemma \ref{ML_estimat}, we have
\[
\|\phi(\cdot,\eta)-\phi(\cdot,\zeta)\|_{\gamma}^2
\leq
\frac{3L^2 (T+1) \Gamma(2\alpha-1)}{\gamma \Gamma(\alpha)^2}\|\phi(\cdot,\eta)-\phi(\cdot,\zeta)\|_{\gamma}^2
+
3\|\eta-\zeta\|_{\rm{ms}}^2.
\]
Thus, by \eqref{coefficient_weight} we have
\begin{equation*}\label{depend_conti}
\left( 1-\frac{{3{L^2}\left( {T + 1} \right)\Gamma \left( {2\alpha  - 1} \right)}}{{ \gamma\Gamma {{\left( \alpha  \right)}^2}}}\right)
\|\phi(\cdot,\eta)-\phi(\cdot,\zeta)\|_{\gamma}^2
\leq
3\|\eta-\zeta\|_{\rm{ms}}^2.
\end{equation*}
Hence,
\[
\lim_{\eta\to \zeta}\sup_{t\in [0,T]}\|\varphi(t,\eta)-\varphi(t,\zeta)\|_{\rm{ms}}=0.
\]
The proof is complete.
\end{proof}
We conclude this section with a discussion on the gap in the proof of global existence of solutions for Caputo fractional stochastic differential equation in \cite{Kloeden_16}.
\begin{remark}\label{Remark1}

For $\alpha\in (\frac{1}{2},1)$, we consider a Caputo fractional stochastic differential equation on a Banach space $X$ of the following form
\begin{equation}\label{tempo_1}
^{C}D^\alpha_{0+}x(t)=b(t,x(t))+\sigma(t,x(t))\frac{dW_t}{dt},
\end{equation}
where $t\in (0,T]$, $b,\sigma:[0,T]\times \mathbb{L}^2 (\Omega;X)\rightarrow \mathbb{L}^2(\Omega;X)$ are measurable functions satisfying the following conditions:
\begin{itemize}
	\item[(i)] there exists a constant $L>0$ such that for all $t\in [0,T]$ and $x,y\in \mathbb{L}^2 $
\[
\mathbb{E}(\|b(t,x)-b(t,y)\|^2)+\mathbb{E}(\|\sigma(t,x)-\sigma(t,y)\|^2)\leq L\mathbb{E}(\|x-y\|^2);
\]
	\item[(ii)] the functions $b,\sigma$ are bounded, i.e. for some $x_0\in \mathbb{L}^2(\Omega;X)$ and $b>0$, there exists a constant $M>0$ such that
	\[
	\mathbb{E}(\|b(t,x)\|^2)\leq M^2,\quad\mathbb{E}(\|\sigma(t,x)\|^2)\leq M^2
	\]
	for all $(t,x)\in R_0:=\{(t,x):0\leq t\leq T, \mathbb{E}(\|x-x_0\|^2)\leq b^2\}$.
\end{itemize}
Using a similar approach as in \cite{Caraballo14, Lakshmikantham08}, the authors in \cite{Kloeden_16} have proven the existence and uniqueness of local solutions on a small interval $[0,T_b]$, where $T_b$ is a parameter depending on $b$ defined as in \cite[Theorem 3.3, p. 209]{Kloeden_16}. Conerning the global existence of solutions, the authors used the method of successive approximation, see \cite[Theorem 3.4, p. 209]{Kloeden_16}. However, this method seems not applicable to fractional stochastic differential equations. More precisely, by the history-dependence of solutions to fractional differential equations, the solutions of the problem
\[
\begin{cases}
D^\alpha x(t)&=b(t,x(t))+\sigma(t,x(t))\frac{dW_t}{dt},\quad t\in [T^*,T^*+\delta),\\
x(T^*)&=x^*\in \mathbb{L}^2(\Omega;X),
\end{cases}
\]
and of the problem
\[
\begin{cases}
D^\alpha x(t)&=b(t,x(t))+\sigma(t,x(t))\frac{dW_t}{dt},\quad t\in [0,T^*+\delta),\\
x(0)&=x_0\in \mathbb{L}^2(\Omega;X)
\end{cases}
\]
are not the same by using a shift of the time.

\end{remark}
\subsection{A lower bound on the asymptotic separation of two distinct solutions}
\begin{proof}[Proof of Theorem \ref{main_result2}]
Suppose a contradiction, i.e. there exists a positive constant $\lambda>\frac{2\alpha}{1-\alpha}$ such that
\begin{equation}\label{A_1}
\limsup_{t\to\infty}\;t^{\lambda}\|\varphi(t,\eta)-\varphi(t,\zeta)\|_{\rm{ms}}<\infty,
\end{equation}
for some $\eta,\zeta\in \frak X_0$, $\eta\ne \zeta$.
Then, there exist constants $T>0$ and $K>0$ such that
\begin{equation}\label{A_2}
\|\varphi(t, \eta)-\varphi(t,\zeta)\|_{\rm{ms}}^2
\leq Kt^{-2\lambda}
\qquad \hbox{for all}\; t\geq T.
\end{equation}
From \eqref{IntegrableForm} and  the inequality $\|x+y+z\|^2\leq 3 (\|x\|^2+ \|y\|^2+ \|z\|^2)$ for all $x,y,z\in \R^d$, we have
\begin{align*}
\|\eta-\zeta\|^2 & \leq 3\|\varphi(t,\eta)-\varphi(t,\zeta)\|^2+\frac{3}{\Gamma(\alpha)^2}\left\|\int_0^t (t-\tau)^{\alpha-1}(\sigma(\tau,\eta)-\sigma(\tau,\zeta))\;dW_\tau\right\|^2\\
&\hspace{1cm}+\frac{3}{\Gamma(\alpha)^2}\left\|\int_0^t (t-\tau)^{\alpha-1}(b(\tau,\varphi(\tau,\eta)-b(\tau,\varphi(\tau,\zeta))\;d\tau\right\|^2.
\end{align*}
Taking the expectation of both sides and using the Ito's isometry, (H1), we obtain  
\begin{eqnarray*}
\|\eta-\zeta\|_{\rm{ms}}^2
&\leq&
3\E(\|\varphi(t, \eta)-\varphi(t,\zeta)\|^2)\\[1ex]
&& +
\frac{3L^2}{\Gamma(\alpha)^2}\mathbb{E}\left(\int_0^t (t-\tau)^{\alpha-1}\|\varphi(\tau,\eta)-\varphi(\tau,\zeta)\|\;d\tau\right)^2\\[1ex]
&&
+ \frac{3L^2}{\Gamma(\alpha)^2}\int_0^t (t-\tau)^{2\alpha-2} \|\varphi(\tau, \eta)-\varphi(\tau,\zeta)\|^2_{\rm{ms}}\;d\tau.
\end{eqnarray*}
From \eqref{A_2}, we derive that $\lim_{t\to\infty} \E(\|\varphi(t, \eta)-\varphi(t,\zeta)\|^2)=0$. Hence, to derive a contradiction and therefore to complete the proof it is sufficient to show that
\begin{equation}\label{Aim1}
\lim_{t\to\infty} I_1(t)=0, \quad\hbox{where }I_1(t):= \mathbb{E}\left(\int_0^t (t-\tau)^{\alpha-1}\|\varphi(\tau,\eta)-\varphi(\tau,\zeta)\|\;d\tau\right)^2
\end{equation}
and
\begin{equation}\label{Aim2}
\lim_{t\to\infty} I_2(t)=0, \quad\hbox{where }I_2(t):=\int_0^t (t-\tau)^{2\alpha-2}\|\varphi(\tau, \eta)-\varphi(\tau,\zeta)\|^2_{\rm{ms}}\;d\tau.
\end{equation}
To prove \eqref{Aim1}, choose and fix $\delta\in \left(\frac{\alpha}{\lambda},\frac{1-\alpha}{2}\right)$.  Note that the existence of such a $\delta$ comes from the fact that  $\frac{\alpha}{\lambda}<\frac{1-\alpha}{2}$ (equivalently, $\lambda>\frac{2\alpha}{1-\alpha}$). For $t>\max\{T^{1/\delta},1\}$, we have
\begin{eqnarray*}
I_1(t)
 &\leq&
 2 \mathbb{E}\left(\int_0^{t^\delta} (t-\tau)^{\alpha-1}\|\varphi(\tau, \eta)-\varphi(\tau,\zeta)\|\;d\tau\right)^2\\
 && + 2  \mathbb{E}\left(\int_{t^\delta}^{t} (t-\tau)^{\alpha-1}\|\varphi(\tau, \eta)-\varphi(\tau,\zeta)\|\;d\tau\right)^2.
\end{eqnarray*}
Using the H\"{o}lder inequality, we obtain 
\begin{eqnarray*}
I_1(t)
&\leq&
2\int_0^{t^\delta} (t-\tau)^{2\alpha-2}\;d\tau
\int_0^{t^\delta} \|\phi(\tau,\eta)-\phi(\tau,\zeta)\|_{\rm{ms}}^2\;d\tau\\
&&
+
2\int_{t^\delta}^t (t-\tau)^{2\alpha-2}\;d\tau
\int_{t^\delta}^t \|\phi(\tau,\eta)-\phi(\tau,\zeta)\|_{\rm{ms}}^2\;d\tau.
\end{eqnarray*}
Since
\[
\int_0^{t^\delta} (t-\tau)^{2\alpha-2}\;d\tau
\leq \frac{t^{\delta}}{(t-t^{\delta})^{2-2\alpha}},
\int_{t^\delta}^t (t-\tau)^{2\alpha-2}\;d\tau
\leq \frac{(t-t^\delta)^{2\alpha-1}}{2\alpha-1}
\]
it follows together with \eqref{A_2} that
\begin{eqnarray*}
I_1(t)
&\leq &
\frac{2 t^{2\delta} \sup_{t\geq 0}\|\phi(\tau,\eta)-\phi(\tau,\zeta)\|_{\rm{ms}}}{(t-t^\delta)^{2-2\alpha}}
+
\frac{2K(t-t^\delta)^{2\alpha-1}}{2\alpha-1} \int_{t^{\delta}}^t \tau^{-2\lambda}\;d\tau \\
&\leq &
\frac{2 t^{2\delta} \sup_{t\geq 0}\|\phi(\tau,\eta)-\phi(\tau,\zeta)\|_{\rm{ms}}}{(t-t^\delta)^{2-2\alpha}}
+
\frac{2K (t-t^{\delta})^{2\alpha}}{ (2\alpha-1)t^{2\delta\lambda}}.
\end{eqnarray*}
By definition of $\delta$, we have $2\delta<2-2\alpha$ and $2\alpha< 2\delta\lambda$. Hence, letting $t\to\infty$ in the preceding inequality yields that $\lim_{t\to\infty} I_1(t)=0$ and thus \eqref{Aim1} is proved. Concerning the assertion \eqref{Aim2}, let $t\geq T$ be arbitrary. By \eqref{A_2}, we have
\begin{eqnarray*}
I_2(t)
&\leq &
\int_0^{T} (t-\tau)^{2\alpha-2} \|\varphi(\tau, \eta)-\varphi(\tau,\zeta)\|^2_{\rm{ms}}\;d\tau
+ K \int_T^{t} (t-\tau)^{2\alpha-2}\tau^{-2\lambda}\;d\tau\\
&\leq &
\frac{T}{(t-T)^{2-2\alpha}} \sup_{t\geq 0} \|\varphi(\tau, \eta)-\varphi(\tau,\zeta)\|^2_{\rm{ms}} +  K \int_T^{t} (t-\tau)^{2\alpha-2}\tau^{-2\lambda}\;d\tau.
\end{eqnarray*}
Therefore,
\begin{equation}\label{A_3}
\limsup_{t\to\infty} I_2(t)
\leq
K \limsup_{t\to\infty}  \int_T^{t} (t-\tau)^{2\alpha-2}\tau^{-2\lambda}\;d\tau.
\end{equation}
Note that for $t\geq 2T$ we have
\begin{eqnarray*}
\int_T^{t} (t-\tau)^{2\alpha-2}\tau^{-2\lambda}\;d\tau
&= &
\int_T^{\frac{t}{2}} (t-\tau)^{2\alpha-2}\tau^{-2\lambda}\;d\tau
+
\int_{\frac{t}{2}}^{t} (t-\tau)^{2\alpha-2}\tau^{-2\lambda}\;d\tau\\
&\leq &
\frac{2^{2-2\alpha}}{t^{2-2\alpha}}\int_T^{\frac{t}{2}}\tau^{-2\lambda}\;d\tau
+\left(\frac{t}{2}\right)^{-2\lambda}\int_{\frac{t}{2}}^t (t-\tau)^{2\alpha-2}\;d\tau\\
&\leq &
\frac{2^{2-2\alpha} T^{-2\lambda+1}}{(2\lambda-1) t^{2-2\alpha}}
+
\frac{1}{2\alpha-1} \left(\frac{t}{2}\right)^{2\alpha-1-2\lambda},
\end{eqnarray*}
which together with \eqref{A_3} and the fact that $\alpha\in\left(\frac{1}{2},1\right), \lambda>\frac{2\alpha}{1-\alpha} >\alpha-\frac{1}{2}$, implies that $\lim_{t\to\infty}I_2(t)=0$. The proof is complete.
\end{proof}
\section*{Acknowledgement}
This research is funded by the Vietnam National Foundation for Science and Technology Development (NAFOSTED) under Grant Number 101.03-2017.01.





%
%
\end{document}